\documentclass{article}

\usepackage{amsthm}
\usepackage{amsmath}
\usepackage{amsfonts}
\usepackage{amssymb}
\usepackage{paralist}
\usepackage{ifthen}
\usepackage{tikz}

\newtheorem{theorem}{Theorem}[section]
\newtheorem{lemma}[theorem]{Lemma}
\newtheorem{proposition}[theorem]{Proposition}
\newtheorem{corollary}[theorem]{Corollary}

\newcommand{\parder}[3][Default]{
	\frac{\partial \ifthenelse{\equal{#1}{Default}}{}{^{#1}}#2}{
              \partial #3 \ifthenelse{\equal{#1}{Default}}{}{^{#1}}}}

\numberwithin{equation}{section}

\makeatletter
\newcommand{\nolisttopbreak}{\vspace{\topsep}\nobreak\@afterheading}
\makeatother
\newenvironment{listproof}[1][\proofname]{\begin{proof}[#1]\mbox{}\nolisttopbreak}{\end{proof}}

\newcommand{\oth}{\textsuperscript{th}\ }
\newcommand{\jac}{{\mathcal{J}}}

\newcommand{\Mat}{\operatorname{Mat}}
\newcommand{\ie}{\operatorname{IE}}
\newcommand{\pe}{\operatorname{PE}}
\newcommand{\GL}{\operatorname{GL}}
\newcommand{\rk}{\operatorname{rk}}
\newcommand{\cork}{\operatorname{cork}}
\newcommand{\N}{\mathbb{N}}

\newcommand{\tr}{\operatorname{tr}}
\newcommand{\trdeg}{\operatorname{trdeg}}
\newcommand{\chr}{\operatorname{char}}
\newcommand{\tp}{^{\rm t}}
\newcommand{\sdots}{\vdots\,\vdots\,\vdots}

\newcommand{\zeromat}{\mbox{\begin{tikzpicture} \useasboundingbox (-0.1,-0.125) -- (0.1,0.125);
\draw (-0.05,0.05) -- node[anchor=center]{\rm 0} (0.05,-0.05); \end{tikzpicture}}}

\begin{document}

\title{Computations of Keller maps over fields with {\mathversion{normal}$\tfrac16$}}

\author{Michiel de Bondt}

\maketitle

\begin{abstract}
\noindent
We classify Keller maps $x + H$ in dimension $n$ over fields with $\tfrac16$,
for which $H$ is homogeneous, and
\begin{enumerate}[\upshape (1)]

\item $\deg H = 3$ and $\rk \jac H \le 2$;

\item $\deg H = 3$ and $n \le 4$;

\item $\deg H = 4$ and $n \le 3$;

\item $\deg H = 4 = n$ and $H_1, H_2, H_3, H_4$ are linearly dependent over $K$.

\end{enumerate}
In our proof of these classifications, we formulate (and prove) several results
which are more general than needed for these classifications. One of these 
results is the classification of all homogeneous polynomial maps $H$ as in (1)
over fields with $\tfrac16$.
\end{abstract}

\section{Introduction}

Let $n$ be a positive integer and let $x = (x_1,x_2,\ldots,x_n)$ be an $n$-tuple
of variables. 
We write $a|_{b=c}$ for the result of substituting $b$ by $c$ in $a$.

Let $K$ be any field. In the scope of this introduction, denote
by $L$ an unspecified (but big enough) field, which contains $K$ or even $K(x)$.

For a polynomial or rational map $H = (H_1,H_2,\ldots,H_m) \in L^m$, write $\jac H$ or 
$\jac_x H$ for the Jacobian matrix of $H$ with respect to $x$. So
$$
\jac H = \jac_x H = \left(\begin{array}{cclc}
\parder{}{x_1} H_1 & \parder{}{x_2} H_1 & \cdots & \parder{}{x_n} H_1 \\
\parder{}{x_1} H_2 & \parder{}{x_2} H_2 & \cdots & \parder{}{x_n} H_2 \\
\vdots & \vdots & \sdots & \vdots \\
\parder{}{x_1} H_m & \parder{}{x_2} H_m & \cdots & \parder{}{x_n} H_m
\end{array}\right)
$$

Denote by $\rk M$ the rank of a matrix $M$, whose entries are contained in $L$,
and write $\trdeg_K L$ for the transcendence degree of $L$ over $K$. 
It is known that $\rk \jac H \le \trdeg_K K(H)$ for a rational map $H$ of any 
degree, with equality if $K(H) \subseteq K(x)$ is separable, in particular if
$K$ has characteristic zero. This is proved in \cite[Th.\@ 1.3]{1501.06046},
see also \cite[Ths.\@ 10, 13]{DBLP:conf/mfcs/PandeySS16}.

Let a \emph{Keller map} be a polynomial map $F \in K[x]^n$, for
which $\det \jac F \in K \setminus \{0\}$. 
If $H \in K[x]^n$ is homogeneous of degree at least $2$, then $x + H$ is
a Keller map, if and only if $\jac H$ is nilpotent. 

We say that a matrix $M \in \Mat_n(L)$ is \emph{similar over $K$}
to a matrix $\tilde{M} \in \Mat_n(L)$, if there exists a $T \in \GL_n(K)$ such that
$\tilde{M} = T^{-1}MT$. If $\jac H$ is similar over $K$ to a triangular 
matrix, say that $T^{-1}(\jac H)T$ is a triangular matrix, then
$$
\jac \big(T^{-1}H(Tx)\big) = T^{-1}(\jac H)|_{x=Tx}T
$$ 
is triangular as well.

All sections focus on \emph{cubic homogeneous polynomial maps}, except the last one, which
is about \emph{quartic polynomial maps}.

\section{Calculations for rank 2}

\begin{theorem} \label{detdep}
Let $s$ and $d$ be positive integers, such that $s \le n$. Take
$$
\tilde{x} := (x_1,x_2,\ldots,x_s) \qquad \mbox{and} \qquad
L := K(x_{s+1},x_{s+2},\ldots,x_n)
$$
In order to prove that for (homogeneous) polynomial maps
$H \in K[x]^m$ of degree $d$,
\begin{equation} \label{rkimptrdeg}
\rk \jac H = r \mbox{ implies } \trdeg_K K(H) = r
\qquad \mbox{for every } r < s
\end{equation}
it suffices to show that for (homogeneous) polynomial maps
$\tilde{H} \in L[\tilde{x}]^s$ of degree $d$,
$$
\det \jac_{\tilde{x}} \tilde{H} = 0 \qquad \mbox{implies} \qquad
f(\tilde{H}) = 0 \mbox{ for some } 
f \in L[y_1,y_2,\ldots,y_s] \setminus \{0\}
$$
\end{theorem}

\begin{proof}
Suppose that $H \in K[x]^m$ is (homogeneous) of degree $d$, such that
\eqref{rkimptrdeg} does not hold. Then there exists an $r < s$ such that
$\rk \jac H = r < \trdeg_K K(H)$.

We first show that we may assume that $\trdeg_K K(H) = s = m$.
For that purpose, take $s' := \trdeg_K K(H)$ and assume without 
loss of generality that $H_1, H_2, \ldots, H_{s'}$ are 
algebraically independent over $K$. Assume again without loss 
of generality that the components of
$$
H' := \big(H_1, H_2, \ldots, H_{s'}, 
      x_{s'+1}^d, x_{s'+2}^d,\ldots,x_s^d\big)
$$      
are algebraically independent over $K$. Since
$$
\rk \jac H' \le r + (s - s') < s = \trdeg_K K(H')
$$
we deduce that \eqref{rkimptrdeg} remains unsatisfied if we replace
$H$ by $H'$. So we may assume that $\trdeg_K K(H) = s = m$.

So assume that the components $H_1, H_2, \ldots, H_s$ of $H$ are algebraically
independent over $K$. One can easily verify that
$$
\deg_{x_1} H_1(x_1,x_1x_2,x_1x_3,\ldots,x_1x_n) = \deg H_1
$$
so $H_1(x_1,x_1x_2,x_1x_3,\ldots,x_1x_n)$ is algebraically
independent over $K$ of $x_2,x_3,\allowbreak\ldots,x_n$.
On account of the Steinitz Mac Lane exchange lemma, we may assume without
loss of generality that the components of
$$
\big(H(x_1,x_1x_2,x_1x_3,\ldots,x_1x_n),x_{s+1},x_{s+2},\ldots,x_n\big)
$$
are algebraically independent over $K$. Then the components of
$H(x_1,x_1x_2,x_1x_3,\allowbreak\ldots,x_1x_n)$ are algebraically independent over 
$L = K(x_{s+1},x_{s+2},\ldots,x_n)$, and so are the components of 
$$
\tilde{H} := H(x_1,x_2,\ldots,x_s, 
x_1x_{s+1},x_1x_{s+2},\ldots,x_1x_n)
$$
Consequently, there does not exist an $f \in L[y_1,y_2,\ldots,y_s] \setminus \{0\}$
such that $f(\tilde{H}) = 0$.

Let $G := (x_1,x_2,\ldots,x_s, x_1x_{s+1},x_1x_{s+2},\ldots,x_1x_n)$.
Then it follows from the chain rule that 
$$
\jac_{\tilde{x}} \tilde{H} = (\jac H)|_{x = G} \cdot \jac_{\tilde{x}} G
$$
so the column space of $\jac_{\tilde{x}} \tilde{H}$ is contained in that
of $(\jac H)|_{x = G}$. From $\rk \jac H < s$, we deduce that 
$\det \jac_{\tilde{x}} \tilde{H} = 0$. 

So $\tilde{H}$ does not satisfy the last line of theorem \ref{detdep}, 
and our reduction is complete.
\end{proof}

\begin{theorem} \label{rkform}
Let $H \in K[x]^m$ be a polynomial map of degree $d$ and $r := \rk \jac H$.
\begin{enumerate}[\upshape (i)]

\item If the cardinality of $K$ exceeds $(d-1)r$ and
$$
\jac H \cdot x = 0
$$
then there are $S \in GL_m(K)$ and $T \in \GL_n(K)$, 
such that for $\tilde{H} := SH(Tx)$,
$$
\tilde{H}|_{x=e_{r+1}} = \left( \begin{array}{cc}
I_r & \zeromat \\ \zeromat & \zeromat                         
\end{array} \right)
$$

\item If the cardinality of $K$ exceeds $(d-1)r + 1$ and
$$
\jac H \cdot x \ne 0
$$
then there are $S \in GL_m(K)$ and $T \in \GL_n(K)$, 
such that for $\tilde{H} := SH(Tx)$,
$$
\tilde{H}|_{x=e_1} = \left( \begin{array}{cc}
I_r & \zeromat \\ \zeromat & \zeromat                         
\end{array} \right)
$$

\end{enumerate}
Furthermore, the cardinality of $K$ may be one less (i.e.\@ at least
$(d-1)r$ or $(d-1)r + 1$ respectively) if every nonzero component of 
$H$ is homogeneous.
\end{theorem}

\begin{listproof}
\begin{enumerate}[\upshape (i)]

\item 
Assume without loss of generality that 
$$
a_0 := \det \jac_{x_1,x_2,\ldots,x_r} (H_1,H_2,\ldots,H_r) \ne 0
$$
Suppose that the cardinality of $K$ exceeds $(d-1)r$.
From \cite[Lemma 5.1 (i)]{1310.7843}, it follows that there exists a
vector $w \in K^n$ such that $a_0(w) \ne 0$.
Hence
$$
r \ge \rk \big(\jac H\big)\big|_{x=w} \ge 
\rk \jac_{x_1,x_2,\ldots,x_r} \big(H_1,H_2,\ldots,H_r\big)\big|_{x=w} = r
$$
so $\rk \big(\jac H\big)\big|_{x=w} = r$. Hence there are exactly $n-r$
independent vectors $v_{r+1}, v_{r+2}, \ldots, \allowbreak v_n \in K^n$,
such that
$
\big(\jac H\big)\big|_{x=w} \cdot v_i = 0
$
for $i = r+1, r+2, \allowbreak \ldots, n$. Since
$$
\big(\jac H\big)\big|_{x=w} \cdot w = \big(\jac H \cdot x\big)\big|_{x=w} = 0
$$
we can take $v_{r+1} = w$.

Take $T \in \GL_n(K)$ of the form $(v_1\,|\,v_2\,|\,\cdots\,|\,v_n)$.
From the chain rule, we deduce that
$$
\big(\jac (H(Tx))\big)\big|_{x=e_{r+1}} \cdot e_i 
= (\jac H)|_{x=Te_{r+1}} \cdot T e_i = (\jac H)|_{x=w} \cdot v_i
$$
for every $i \le n$. In particular, 
$\rk \big(\jac (H(Tx))\big)\big|_{x=e_{r+1}} = r$ and the last
$n - r$ columns of $\big(\jac (H(Tx))\big)\big|_{x=e_{r+1}}$ are zero. 
So we can take $S \in \GL_m(K)$ such that
$$
\big(\jac (SH(Tx))\big)\big|_{x=e_{r+1}} =
S \cdot \big(\jac (H(Tx))\big)\big|_{x=e_{r+1}} =
\left( \begin{array}{cc}
I_r & \zeromat \\ \zeromat & \zeromat                         
\end{array} \right)
$$

\item Suppose that the cardinality of $K$ exceeds $(d-1)r + 1$.
Since $\jac H \cdot x \ne 0$, we may assume that
$$
\rk \Big( \jac H \cdot x \,\Big|\, 
\jac_{x_2,x_3,\ldots,x_r} H \Big) = r
$$
and that 
$$
a_1 := \det \Big( \jac (H_1,H_2,\ldots,H_r) \cdot x \,\Big|\, 
\jac_{x_2,x_3,\ldots,x_r} (H_1,H_2,\ldots,H_r) \Big) \ne 0
$$
From \cite[Lemma 5.1 (i)]{1310.7843}, it follows that there exists a
vector $w \in K^n$ such that $a_1(w) \ne 0$. 

Just as in the proof of (i), $\rk \big(\jac H\big)\big|_{x=w} = r$ 
and there are independent vectors $v_{r+1}, v_{r+2}, \ldots, v_n \in K^n$,
such that $\big(\jac H\big)\big|_{x=w} \cdot v_i = 0$
for $i = r+1, \allowbreak r+2, \ldots, n$. Since 
$\big(\jac H \cdot x\big)\big|_{x=w}$ is the first column of a matrix
whose determinant is nonzero, we deduce that
$$
\big(\jac H\big)\big|_{x=w} \cdot w = \big(\jac H \cdot x\big)\big|_{x=w} \ne 0
$$
so $v_1 := w$ is independent of $v_{r+1}, \allowbreak v_{r+2}, \ldots, v_n$.

Take $T \in \GL_n(K)$ of the form $(v_1\,|\,v_2\,|\,\cdots\,|\,v_n)$.
From the chain rule, we deduce that
$$
\big(\jac (H(Tx))\big)\big|_{x=e_1} \cdot e_i 
= (\jac H)|_{x=Te_1} \cdot T e_i = (\jac H)|_{x=w} \cdot v_i
$$
for every $i \le n$. The rest of the proof of (ii) is similar to
that of (i). 

\end{enumerate}
The last claim follows from \cite[Lemma 5.1 (ii)]{1310.7843}, as an
improvement to \cite[Lemma 5.1 (i)]{1310.7843}.
\end{listproof}

\begin{lemma} \label{trdegle2}
Let $H \in K[x]^m$ be cubic homogeneous and suppose that $\trdeg_K \allowbreak 
K(H) = r \le 2$. 

Then there are $S \in GL_m(K)$ and $T \in \GL_n(K)$, 
such that for $\tilde{H} := SH(Tx)$, one of the following holds:
\begin{enumerate}[\upshape (1)]

\item $\tilde{H}_{r+1} = \tilde{H}_{r+2} = \cdots = \tilde{H}_m = 0$,

\item $r=2$ and $\tilde{H} \in K[x_1,x_2]^m$,

\item $r=2$ and $K \tilde{H}_1 + K \tilde{H}_2 + \cdots + K \tilde{H}_m
= K x_3 x_1^2 \oplus K x_3 x_1 x_2 \oplus K x_3 x_2^2$.

\end{enumerate}
Furthermore, we can take $S = T^{-1}$ if $m = n$.
\end{lemma}

\begin{proof}
If $m = n$ and $\tilde{H}$ is as in (1), then we can replace $T$ by $S^{-1}$,
without affecting (1). If $m = n$ and $\tilde{H}$ is as in (2) or (3), then 
we can replace $S$ by $T^{-1}$. This proves the last claim. 

From the fact that $H$ is homogeneous of positive degree, it follows that 
$\trdeg_K K(tH) = r$  as well. Suppose first that $r \le 1$. 
From \cite[Th.\@ 2.7]{1501.06046}, it follows that we can take $\tilde{H}$ as in (1).

Suppose next that $r = 2$. From \cite[Th.\@ 2.7]{1501.06046}, 
it follows that $H$ is of the form $g \cdot h(p,q)$, 
such that $g$, $h$ and $(p,q)$ are homogeneous and 
$\deg g + \deg h \cdot \allowbreak \deg (p,q) = 3$.
Assume without loss of generality that $\deg h \le 0$ if $\deg (p,q) \le 0$.
Then $\deg h \le 3$.

If $\deg h \le 1$, then every triple of components of $h$ is linearly dependent
over $K$, so we can take $\tilde{H}$ as in (1). 
If $\deg h = 3$, then $\deg (p,q) = 1$ and $\deg g = 0$, so we can take 
$\tilde{H}$ as in (2).

So assume that $\deg h = 2$. Then $\deg(p,q) = 1$ and $\deg g = 1$. If
$g$ is a linear combination of $p$ and $q$, then we can take $\tilde{H}$ 
as in (2). So assume that $g$ is not a linear combination of $p$ and $q$.
Then we can take $\tilde{H}$ as in (3) or (1).
\end{proof}

\begin{theorem} \label{rkle2}
Suppose that $\frac16 \in K$ and let $H \in K[x]^m$ be cubic
homogeneous. Define $r := \rk \jac H$ and suppose that $1 \le r \le 2$.

Then there are $S \in GL_m(K)$ and $T \in \GL_n(K)$, 
such that for $\tilde{H} := SH(Tx)$, one of the following holds:
\begin{enumerate}[\upshape(1)]

\item only the first $r$ rows of $\jac \tilde{H}$ are nonzero;

\item $r=2$ and only the first $2$ columns of $\jac \tilde{H}$ are nonzero;

\item $r=2$ and $K \tilde{H}_1 + K \tilde{H}_2 + \cdots + K \tilde{H}_m
= K x_3 x_1^2 \oplus K x_3 x_1 x_2 \oplus K x_3 x_2^2$.

\end{enumerate}
Furthermore, we can take $S = T^{-1}$ if $m = n$.
\end{theorem}

\begin{proof}
From lemma \ref{trdegle2}, we deduce that it suffices to show
that $r = \trdeg_K \allowbreak K(H)$. From theorem \ref{detdep},
it follows that we may assume that $m = n = 3$, 
and that it suffices to show that
\begin{equation} \label{detdep3}
\det \jac H = 0 \qquad \mbox{implies} \qquad
f(H) = 0 \mbox{ for some } 
f \in K[y_1,y_2,y_3] \setminus \{0\}
\end{equation}
Since we can replace $K$ by an extension field of $K$ to make it large
enough, it follows from theorem \ref{rkform} that we may assume that 
$$
\big(\jac H\big)\big|_{x=(1,0,0)}
= \left( \begin{array}{ccc}
1 & 0 & 0 \\ 0 & 1 & 0 \\ 0 & 0 & 0
\end{array} \right)
$$
What remains is a calculation, which we have performed with Maple 8, see
{\tt dim3detcub.pdf}. It appeared that \eqref{detdep3} was valid.
\end{proof}

\section{Rank 2 with nilpotency}

\begin{theorem} \label{trdeg1}
Let $F = x + H$ be a Keller map, such that $\trdeg_K K(H) = 1$. 

Then $\jac H$ is similar over $K$ to a triangular matrix, and the following
statements are equivalent:
\begin{enumerate}[\upshape (1)]

\item $\det \jac F = 1$;

\item $\jac H$ is nilpotent;

\item $(\jac H) \cdot (\jac H)|_{x=y} = 0$.

\end{enumerate}
\end{theorem}

\begin{proof} 
From \cite[Cor.\@ 3.2]{1501.06046}, it follows that there exists a polynomial 
$p \in K[x]$ such that $H_i \in K[p]$ for every $i$. Say that $H_i = h_i(p)$,
where $h_i \in K[t]$ for each $i$. Write $h_i' = \parder{}{t} h_i$,
then
\begin{equation} \label{JHhp}
\jac H = h'(p) \cdot \jac p
\end{equation}
Assume without loss of generality that
$$
h_1' = h_2' = \cdots = h_s' = 0
$$
and
$$
0 \le \deg h_{s+1}' < \deg h_{s+2}' < \cdots < \deg h_n'
$$
If $s < i < n$, then 
$$
\deg h_i'(p) = \deg h_i' \cdot \deg p 
\le (\deg h_{i+1}' - 1) \cdot \deg p = \deg h_{i+1}'(p) - \deg p
$$
Since the degrees of the entries of $\jac p$ are less than $\deg p$, we deduce from
\eqref{JHhp} that the nonzero entries on the diagonal of $\jac H$
all have different degree, in increasing order. Furthermore, the 
nonzero entries beyond the $(s+1)$\oth entry on the diagonal of $\jac H$ have positive
degrees.

From \eqref{JHhp}, it follows that $\rk (- \jac H) \le 1$. 
Consequently, $n-1$ eigenvalues of $-\jac H$ are zero. It follows that
the trailing degree of the eigenvalue polynomial of $-\jac H$ is at least
$n-1$. More precisely,
$$
\det (t I_n - -\jac H) = t^n - \tr (-\jac H) \cdot t^{n-1}
$$
so
$$
\det \jac F = \big(t^n - \tr (-\jac H) \cdot t^{n-1}\big)\big|_{t=1} = 1 + \tr \jac H
$$
We infer that the diagonal of $\jac H$ is totally zero, except maybe the 
$(s+1)$\oth entry, which is constant. 

So $\parder{}{x_i} p = 0$ for all $i > s+1$, and $\jac H$ is lower triangular.
If the $(s+1)$\oth entry on the diagonal of $\jac H$ is nonzero, then
(1), (2) and (3) do not hold. So assume that the $(s+1)$\oth entry on the 
diagonal of $\jac H$ is zero. Then $\parder{}{x_i} p = 0$ for all $i > s$, and
(1), (2) and (3) do hold.
\end{proof}

\begin{lemma} \label{2x2}
Let $N \in \Mat_2(K)$ such that $N$ is nilpotent. Then there
are $a,b,c \in K[x]$, such that
$$
N = c \left( \begin{array}{cc} 
ab & -b^2 \\ a^2 & -ab \end{array}\right)
$$
Furthermore, $N$ is similar over $K$ to a triangular matrix, if and only
if $a$ and $b$ are linearly dependent over $K$.
\end{lemma}

\begin{proof}
Since $\det N = 0$, we can write $N$ in de form
$$
N = c \cdot \binom{b}{a} \cdot \big(\,a ~~ {-\tilde{b}}\,\big)
$$
where $a,b \in K[x]$ and $\tilde{b},c \in K(x)$. From $\tr N = 0$,
it follows that $\tilde{b} = b$. If we choose $a$ and $b$ relatively prime,
then $c \in K[x]$ as well.

Furthermore, $a$ and $b$ are linearly dependent over $K$, 
if and only if the rows of $N$ are linearly dependent over $K$, 
if and only if $N$ is similar over $K$ to a triangular matrix.
\end{proof}

\begin{lemma} \label{2x2nilp}
Let $H \in K[x]^2$ be cubic homogeneous, such that $\jac_{x_1,x_2} H$ is 
nilpotent. Then there exists a $T \in \GL_2(K)$ such that for 
$\tilde{H} := T^{-1} H\big(T(x_1,x_2),x_3,\allowbreak x_4,\allowbreak 
\ldots,x_n\big)$, one of the following holds:
\begin{enumerate}[\upshape (1)]

\item $\jac_{x_1,x_2} \tilde{H}$ is a triangular matrix;

\item there are independent linear forms $a, b \in K[x]$,
such that
$$
\jac_{x_1,x_2} \tilde{H} = 
\left( \begin{array}{cc} ab & -b^2 \\ a^2 & -ab \end{array} \right)
\qquad \mbox{ and } \qquad 
\jac_{x_1,x_2} \left( \begin{array}{c} a \\ b \end{array} \right)
=  \left( \begin{array}{cc} 0 & 0 \\ 0 & 0 \end{array} \right)
$$

\item $\frac13 \notin K$ and there are independent linear forms $a, b \in K[x]$,
such that
$$
\jac_{x_1,x_2} \tilde{H} = 
\left( \begin{array}{cc} ab & -b^2 \\ a^2 & -ab \end{array} \right)
\qquad \mbox{and} \qquad 
\jac_{x_1,x_2} \left( \begin{array}{c} a \\ b \end{array} \right)
=  \left( \begin{array}{cc} 0 & 1 \\ 1 & 0 \end{array} \right)
$$
\end{enumerate}
\end{lemma}

\begin{proof}
Suppose that (1) does not hold.
From lemma \ref{2x2}, it follows that there are $a,b,c \in K[x]$, such that
$$
\jac_{x_1,x_2} H = c \left( \begin{array}{cc}
ab & - b^2 \\ a^2 & -ab \end{array} \right)
$$
where $a$ and $b$ are linearly independent over $K$. As $H$ is cubic
homogeneous, the entries of $\jac_{x_1,x_2} H$ are quadratic homogeneous,
so $c \in K$ and $a$ and $b$ are independent linear forms.

If we take
$$
T = \left( \begin{array}{cc} c & 0 \\ 0 & 1 \end{array} \right)
$$
then it follows from the chain rule that
$$
\jac_{x_1,x_2} \tilde{H} = \left( \begin{array}{cc}
\tilde{a}\tilde{b} & -\tilde{b}^2 \\ \tilde{a}^2 & -\tilde{a}\tilde{b} \end{array} \right)
$$
where $\tilde{a} = c \cdot a|_{x_1=cx_1}$ and $\tilde{b} = b|_{x_1=cx_1}$.

We show that the coefficient of $x_2$ in $\tilde{b}$ is $0$. So assume the opposite.
Then the coefficient of $x_2^3$ in 
$$
3 \tilde{H}_1 = \jac_{x_1,x_2} \tilde{H}_1 \cdot \binom{x_1}{x_2}
$$
is nonzero. In particular, $\frac13 \in \chr K$. 
Let $\tilde{c} := \tilde{a} + \tilde{b} \parder{}{x_1} \tilde{b}$.
Then $\tilde{c} \tilde{b} = \tilde{a} \tilde{b} + \parder{}{x_1} (\frac13 \tilde{b}^3)$
and
$$
\jac_{x_1,x_2} (\tilde{H}_1 + \tfrac13 \tilde{b}^3) = (\tilde{c}\tilde{b}~0) 
$$
As a consequence, $\parder{}{x_2} \tilde{c}\tilde{b} = \parder{}{x_1} 0 = 0$. Furthermore,
$\tilde{c}$ and $\tilde{b}$ are independent, just like $\tilde{a}$ and 
$\tilde{b}$.

From $\parder{}{x_2} \tilde{c}\tilde{b} = 0$, it follows that $\tilde{c}\tilde{b} \in 
K[x_1,x_3,x_4,\ldots,x_n]$ if $\frac12 \in K$. Since $\tilde{c}$ 
and $\tilde{b}$ are independent, we deduce that $\tilde{c}\tilde{b} \in 
K[x_1,x_3,x_4,\ldots,x_n]$ if $\frac12 \notin K$ as well. 
As the coefficient of $x_2$ in $\tilde{b}$ is nonzero, we must have 
$\tilde{c} = 0$, which is a contradiction. 

So the coefficient of $x_2$ in $\tilde{b}$ is $0$. Similarly, the 
coefficient of $x_1$ in $\tilde{a}$ is $0$, so
$$
\jac_{x_1,x_2} \left( \begin{array}{c} \tilde{a} \\ \tilde{b} \end{array} \right)
=  \left( \begin{array}{cc} 0 & \lambda \\ \mu & 0 \end{array} \right)
$$
where $\lambda, \mu \in K$, and
$$
\jac_{x_1,x_2} \tilde{H} = \left( \begin{array}{cc}
(\mu x_1 + \cdots) (\lambda x_2 + \cdots) & -(\mu x_1 + \cdots)^2 \\ 
(\lambda x_1 + \cdots)^2 & -(\mu x_1 + \cdots) (\lambda x_2 + \cdots)
\end{array} \right)
$$
If $\frac12 \notin K$, then $\lambda \mu = 0$ because antidifferentiating 
$x_1 x_2$ with respect to $x_1$ is not possible.
So the coefficient of $x_1^2 x_2$ in $2\tilde{H}_1$ is equal to both
$\lambda\mu$ and $-2\mu^2$, regardsless of whether $\frac12 \in K$ or
not. Similarly, the coefficient of $x_1 x_2^2$ in 
$2\tilde{H}_2$ is equal to both $\lambda\mu$ and $-2\lambda^2$. 

So either
$\lambda = \mu = 0$ or $\lambda = -2\mu = 4\lambda$. In the latter
case, $\frac13 \notin K$ and $\lambda = \mu$. We can get rid of $\lambda$ 
and $\mu$ if we replace $\tilde{H}$ by 
$\lambda \tilde{H}\big(\lambda^{-1}(x_1,x_2),x_3,x_4,\ldots,x_n\big)$.
\end{proof}

\begin{theorem} \label{uporkle2}
Suppose that $\frac16 \in K$.
Let $H \in K[x]^n$, such that $H$ is cubic homogeneous and $\jac H$ is nilpotent.
\begin{enumerate}[\upshape (i)]

\item  
If $\rk \jac H = 1$, then there exists a $T \in \GL_n(K)$ such that for 
$\tilde{H} := T^{-1} H(Tx)$,
\begin{align*}
\tilde{H}_1 &\in K[x_2,x_3,x_4,\ldots,x_n]  \\
\tilde{H}_2 &= \tilde{H}_3 = \tilde{H}_4 = \cdots = \tilde{H}_n = 0
\end{align*}

\item
If $\rk \jac H = 2$ and $\jac H$ is not similar over $K$ to a 
triangular matrix, then there exists a $T \in \GL_n(K)$ such that for 
$\tilde{H} := T^{-1} H(Tx)$,
\begin{align*}
\tilde{H}_1 &- (x_1x_3x_4-x_2x_4^2) \in K[x_3,x_4,\ldots,x_n]  \\
\tilde{H}_2 &- (x_1x_3^2-x_2x_3x_4) \in K[x_3,x_4,\ldots,x_n]  \\
\tilde{H}_3 &= \tilde{H}_4 = \cdots = \tilde{H}_n = 0
\end{align*}

\end{enumerate}
Furthermore, $x+H$ is invertible if $\rk \jac H \le 2$. Moreover, $x + H$ is 
even tame if either $\rk \jac H = 1$ or $\rk \jac H = 2$ and $n \ge 5$.
\end{theorem}

\begin{proof}
We can take $\tilde{H}$ as in (1), (2) or (3) of theorem \ref{rkle2}. 
If $\rk \jac H = 1$, then $\jac \tilde{H}$ is as in (1) of theorem \ref{rkle2}, 
and (i) holds, because $\tr \jac \tilde{H} = 0$. So assume that $\rk \jac H = 2$. 

If $\tilde{H}$ is as in (3) of theorem \ref{rkle2}, then $\tilde{H}_3 = 0$,
because $x_3^{-1} \tilde{H}_3$ is the constant part with respect to 
$x_3$ of $\tr \jac \tilde{H} = 0$. It follows that in all three cases 
of theorem \ref{rkle2}, $\jac_{x_1,x_2}(\tilde{H}_1,\tilde{H}_2)$ is nilpotent.

Furthermore, $\jac_{x_1,x_2}(\tilde{H}_1,\tilde{H}_2)$ is similar over 
$K$ to a triangular matrix, if and only $\jac \tilde{H}$
is similar over $K$ to a triangular matrix, if and only if 
$\jac H$ is similar over $K$ to a triangular matrix.

Now suppose that $\jac H$ is not similar over $K$ to a
triangular matrix. Then $\jac_{x_1,x_2}(\tilde{H}_1,\tilde{H}_2)$
is not similar over $K$ to a triangular matrix. Since $\frac16 \in K$,
and $x_3^2 \nmid ab$ in (2) of lemma \ref{2x2nilp},
it follows from lemma \ref{2x2nilp} that $\tilde{H}_1 \notin 
K[x_1,x_2,x_3]$, so $\tilde{H}$ is not as in (2) or (3) of theorem 
\ref{rkle2}.

So $\tilde{H}$ is as in (1) of theorem \ref{rkle2}, and
$\tilde{H}_3 = \tilde{H}_4 = \cdots = \tilde{H}_n = 0$. 
Consequently, we can take $\tilde{H}$ such that $a = x_3$
and $b = x_4$ in (2) of lemma \ref{2x2nilp}. So (ii) holds.

To prove the last claim, notice first that $x+H$ is tame if
$\jac H$ is similar over $K$ to a triangular matrix. 
Now suppose that $\jac H$ is not similar over $K$ 
to a triangular matrix and $n \ge 5$. Then $x+H$ is tame,
if and only if $x + \tilde{H}$ is tame, which is the case
if
$$
\big(x_1 + x_4(x_3x_1-x_4x_2),x_2 + x_3(x_3x_1-x_4x_2),x_3,x_4,x_5\big)
$$
is tame. Using extension of scalars, we see that in order to prove that
$x+H$ is tame, it suffices to show that
$$
\big(x_1 + cb(ax_1-bx_2),x_2 + ca(ax_1-bx_2),x_3\big)
$$
is tame as a polynomial map in dimension $3$ over $K[a,b,c]$.
This is done in lemma \ref{tamelm} below.
\end{proof}

\begin{lemma}\label{tamelm}
\begin{align*}
\lefteqn{\big(x_1 + cb(ax_1-bx_2),x_2 + ca(ax_1-bx_2),x_3\big)} \\
&= \big(x_1 + b c x_3, x_2 + a c x_3, x_3\big) \circ 
   \big(x_1, x_2, x_3 + (a x_1 - b x_2)\big) \circ {} \\
&\quad~ \big(x_1 - b c x_3, x_2 - a c x_3, x_3\big) \circ 
   \big(x_1, x_2, x_3 - (a x_1 - b x_2)\big)
\end{align*}
\end{lemma}

One can verify lemma \ref{tamelm} with Maple or something, or use
the proposition in \cite{MR1001475}, with 
$$
D = cb\parder{}{x_1} + ca\parder{}{x_2} \qquad \mbox{and} \qquad  
a x_1 + b x_2 \in \ker D
$$
to get a proof.

\section{Nilpotent Jacobians and computation}

For nilpotent matrices, the conjugation classes are given by Jordan normal forms. 
Now it would be useful to have a similar reduction by linear conjugations for 
non-linear maps with nilpotent Jacobians. Notice that for maps of degree $d$, the
Jacobian has degree $d-1$, and linear conjugation do not change this. So it is impossible
to get a Jordan normal form by linear conjugations of maps of degree $2$ at least.

But one can substitute some constant vector in $x$ in the Jacobian and hope that the 
Jacobian will be a Jordan normal form after this substitution. We will show that
this is indeed possible after a suitable linear conjugation, provided the base field
is infinite. Furthermore, we can obtain that the substitution vector is the sum of at 
most $\sqrt{n}$ distinct unit vectors. 

For a matrix $M \in \Mat_n(K)$, write $\cork M := n - \rk M$.
Let $v \in K^n$ be nonzero and $M \in \Mat_n(K)$ be nilpotent.
Define the {\em image exponent} of $v$ with respect to $M$ as
$$
\ie (M,v) = \ie_K (M,v) := \max \{i \in \N \mid M^i v \ne 0\}
$$ 
and the {\em preimage exponent} of $v$ with respect to $M$ as
$$
\pe (M,v) = \pe_K (M,v) := \max \{i \in \N \mid M^i w = v \mbox{ for some } w \in K^n\}
$$

Suppose that $N \in \Mat_{n}(K)$ has ones on the subdiagonal and zeros
elsewhere. Then $\ie (N,v) + \pe (N,v) = n-1$ for each nonzero $v \in K^n$.
If $M \in \Mat_n(K(x))$ is nilpotent and $\cork M = 1$, 
then $N = T^{-1} M T$ for some $T \in \GL_n(K(x))$.
So if $M$ is nilpotent and $\cork M = 1$, then
$\ie (M,v) + \pe (M,v) = n-1$ for each nonzero $v \in K(x)$ as well, 
because $N$ is the Jordan normal form of $M$, 
$\ie(T^{-1}MT,T^{-1} v) = \ie(M,v)$ and $\pe(T^{-1}MT,T^{-1} v) = \pe(M,v)$.

\begin{proposition} \label{Jordanadd}
Assume $M \in \Mat_n(K)$ is nilpotent and $v \in K^n$ is nonzero.
Then there exists a $T \in \GL_n(K)$ such that
$N := T^{-1} M T$ is the Jordan normal form of $M$ and
$$
w := T^{-1} v = e_{i_1} + e_{i_2} + \cdots + e_{i_m}
$$
where
$$
\ie (N,e_{i_1}) < \ie (N,e_{i_2}) < \cdots < \ie (N,e_{i_m}) = \ie (N,w) = \ie (M,v)
$$
and
$$
\pe (M,v) = \pe (N,w) = \pe (N,e_{i_1}) < \pe (N,e_{i_2}) < \cdots < \pe (N,e_{i_m})
$$
\end{proposition}

\begin{proof}
We distinguish three cases:
\begin{itemize}

\item $\cork M = 1$. \\
Let $N$ be the matrix with ones on the subdiagonal and zeros elsewhere.
Then $N$ is the Jordan normal form of $M$. All eigenvalues of $M$ are zero,
in particular contained in $K$, so $N = T^{-1} M T$ for some $T \in \GL_n(K)$.
Let $w := T^{-1} v$ and $i$ be the index of the first nonzero coordinate of $w$.
Notice that $n-i = \ie(N,w) = \ie(M,v)$ and $i-1 = \pe(N,w) = \pe(M,v)$.

The operator $x \mapsto Nx$ shifts the coordinates of its argument one step downward,
inserting a zero above. The operator $x \mapsto N\tp x$ shifts the coordinates of its 
argument one step upward, inserting a zero below. Now define
the matrix $S \in \GL_N(K)$ by $S e_i = w$, $S e_{i+j} = N^j w$ and 
$S e_{i-j} = (N\tp)^j w$, for all $j \ge 1$. Then $(TS)^{-1} v = S^{-1}w = e_i$, 
so it suffices to show that $(TS)^{-1} M TS = S^{-1} N S$ is the Jordan normal form of
$M$. Indeed $S^{-1} N S e_j = N e_j$ for all $j$, 
because by definition of $i$, $S$ is constructed in such a way that
$N S e_j = S e_{j+1}$ for all $j < n$ and $N S e_n = 0$.

\item $\cork M = 2$. \\
Again, let $N = T^{-1} M T$ be the subdiagonal 
Jordan normal form of $M$ and $w = T^{-1} v$. Notice that $N$ has
two Jordan blocks, say $N_1 \in \Mat_r(K)$ and $N_2 \in \Mat_{n-r}(K)$.
Since $\cork N_1 = \cork N_2 = 1$, it follows from the case 
$\cork M = 1$ that we may assume that
$w$ is the sum of at most two unit vectors $e_i$ and $e_j$, such that
$1 \le i \le r < j \le n$. If $w = e_i$ or $w = e_j$, then we are done, so assume
$w = e_i + e_j$.

Assume without loss of generality that 
$\pe (N,e_i) \le \pe (N,e_j)$ and, in case $\pe (N,e_i) = \pe (N,e_j)$, that
$\ie (N,e_i) \ge \ie (N,e_j)$. Since
we are done in case both $\ie (N,e_i) < \ie (N,e_j)$ and $\pe (N,e_i) < \pe (N,e_j)$,
we may assume that $\ie (N,e_i) \ge \ie (N,e_j)$ in any case.

Since $\ie (N,e_j) \le \ie (N,e_i) = r-i$ it follows that $\ie(N,w) = r-i$.
Since $\pe (N,e_j) \ge \pe (N,e_i) = i-1$ it follows that $\pe(N,w) = i-1$. 
In fact, $N^{i-1} (e_1 + e_{j-i+1}) = w$.

Now define
the matrix $S \in \GL_n(K)$ by $S e_k = e_k + e_{j-i+k}$ if $j-i+k \le n$ and
$S e_k = e_k$ if $j-i+k > n$. Then $S e_i = e_i + e_j = w$, so
$S^{-1} w = e_i$. From $n-j = \ie (N,e_j) \le \ie (N,e_i) = r-i$,
we obtain that $j - i + r \ge n$. Consequently, $Se_r \in \{e_r+e_n,e_r\}$ and
$N S e_r = 0$.

Since $N S e_k = S e_{k+1}$
for all $k \not\in \{r, n\}$ and $N S e_r = 0 = N S e_n$,
it follows that $S^{-1} N S = N$. So we can get rid of $e_j$ as a summand of $w$.
This gives the desired result.

\item $\cork M \ge 3$. \\
Again, let $N = T^{-1} M T$ be the subdiagonal 
Jordan Normal Form of $M$ and $w = T^{-1} v$. From the case $\cork M = 1$, we obtain
that we may assume that $w$ is the sum of at most one unit vector $e_i$
for each Jordan block. From the case $\cork M = 2$, we obtain that we may assume that
two summands $e_i$ and $e_j$ of $w$ satisfy $\ie (N,e_i) < \ie (N,e_j)$ and
$\pe (N,e_i) < \pe (N,e_j)$. That gives the desired result. \qedhere

\end{itemize}
\end{proof}

Notice that $m$ in proposition \ref{Jordanadd} is at most $\sqrt{n}$. This is
because the size of the Jordan block with coordinate $i_{k+1}$ must be at least $2$ larger
than that with $i_k$ (in order to have both $\ie (N,e_{i_{k}}) < \ie (N,e_{i_{k+1}})$ and
$\pe (N,e_{i_{k}}) < \pe (N,e_{i_{k+1}})$) so the sizes are at least $1, 3, 5, \ldots, 2m-1$,
and the series of the odd numbers are the squares.

\begin{theorem} \label{Jordanth}
Let $K$ be an infinite field. 
Take $H \in K[x]^n$ of degree at most $d$, such that $\jac H$
is nilpotent.

Then there exists a $T \in \GL_n(K)$ such that 
$$
\Big(\jac\big(T^{-1} H(Tx)\big)\Big)\Big|_{x = w} = N
$$
where $N$ is the Jordan Normal Form of $\jac H$ and 
$$
w = e_{i_1} + e_{i_2} + \cdots + e_{i_m}
$$
such that
$$
\ie (N,e_{i_1}) < \ie (N,e_{i_2}) < \cdots < \ie (N,e_{i_m}) = \ie_{K(x)} (\jac H,x)
$$
and
$$
\pe_{K(x)} (\jac H, x) = \pe (N,e_{i_1}) < \pe (N,e_{i_2}) < \cdots < \pe (N,e_{i_m})
$$
\end{theorem}

\begin{proof}
Since $\jac H$ is nilpotent, all eigenvalues of $\jac H$ are zero,
in particular contained in $K(x)$, so there exists an 
$S \in \GL_n\big(K(x)\big)$, such that $S^{-1}(\jac H)S$ has lower 
triangular Jordan Normal Form. For the $i${\textsuperscript th} column
$S e_i$ of $S$, we have
$$
(\jac H) S e_i = S \big(S^{-1} (\jac H) S\big) e_i \in \{S e_{i+1}, 0\}
$$
Furthermore, we can write $x$ as a $K(x)$-linear combination of the 
columns of $S$.

Let $I$ be the set of column indices of $S$, for which the coefficients
of $x$, as a $K(x)$-linear combination of the columns of $S$, are nonzero.
Take $v \in K^n$, such that $S|_{x=v} \in \GL_n(K)$, and such that `$|_{x=v}$ 
preserves $I$', i.e.\@ $I$ is also the set of column indices of 
$S|_{x=v}$, for which the coefficients 
of $v = x|_{x=v}$, as a $K$-linear combination of the columns of $S|_{x=v}$, 
are nonzero.

Then 
$$
S|_{x=v}^{-1} (\jac H)|_{x=v} S|_{x=v} = \big(S^{-1} (\jac H) S\big)\big|_{x=v} = 
S^{-1} (\jac H) S
$$
because $S^{-1} (\jac H) S \in \Mat_n(K)$. Furthermore,
$M := (\jac H)|_{x=v}$ has the same Jordan Normal Form as $\jac H$, 
$\ie(M,v) = \ie_{K(x)}(\jac H,x)$ and $\pe(M,v) = \pe_{K(x)}(\jac H,x)$. 
So there exists a $T \in \GL_n(K)$ such that $N := T^{-1} M T$ 
and $w := T^{-1} v$ satisfy the properties of proposition
\ref{Jordanadd}.

Now
\begin{eqnarray*}
\Big(\jac \big(T^{-1}H(Tx)\big)\Big)\Big|_{x=w} 
&=& T^{-1} (\jac H)|_{x=Tx}|_{x=w} T \\ 
&=& T^{-1} (\jac H)|_{x=v} T \\ &=& T^{-1} M T \\ &=& N
\end{eqnarray*}
as desired.
\end{proof}

\begin{corollary} \label{Jordancor}
Let $K$ be an infinite field, and 
$H \in K[x]^n$, such that $\jac H^s = 0$, but $\jac H^{s-1} \cdot x \ne 0$. 

Then
there exists a $T \in \GL_n(K)$, such that for $\tilde{H} := T^{-1}H(Tx)$,
we have that $(\jac \tilde{H})|_{x=e_1}$ has lower triangular Jordan Normal 
Form, with a Jordan block of size $s$ in the upper left corner.
\end{corollary}

\begin{proof}
Since $\jac H^{s-1} \cdot x$ is nonzero and $\jac H^s \cdot x = 0$, 
we have $\ie_{K(x)}(\jac H, x) = s-1$. So $\ie(N, e_{i_m}) = s-1$. 
Consequently, $e_{i_m}$ is the $s$\oth standard basis unit vector from below 
in the range of standard basis unit vectors which coincide with $J$, where 
$J$ is the lower triangular Jordan block of $N$ which coincides with $e_{i_m}$. 
In particular, the size of $J$ is at least $s$. 

Since $N^s = 0$ along with  $\jac H^s = 0$, it follows that $N$ has no Jordan 
block whose size exceeds $s$. So $J$ has size $s$ and $J$ is the largest Jordan 
block of $N$. Furthermore, $e_{i_m}$ is the first standard basis unit vector 
in the range of $s$ standard basis unit vectors which coincide with $J$. 
Hence $\pe(N, e_{i_m}) = 0$. So $m = 1$ and $\pe_{K(x)}(JH, x) = 0$. 
We can permute $J$ to the upper left corner of $N$, which gives the desired
result. 
\end{proof}

\begin{theorem}[{inspired by \cite[Lem.\@ 2.10]{MR3178045}}]
Let $H \in K[x]^n$ be quadratic homogeneous. Then 
$$
\pe_{K(x)}(\jac H, x) = 0 \qquad \mbox{and} \qquad 
\pe \big((\jac H)|_{x=v}, v\big) = 0 \mbox{ for all nonzero } v \in K^n
$$
\end{theorem}

\begin{proof}[Proof ({following the proof of \cite[Lem.\@ 2.10]{MR3178045}})]
It suffices to prove that 
$$
\pe_{K(x)}\big((\jac H)|_{x=v}, v\big) = 0 \mbox{ for all nonzero } v \in K(x)^n
$$
So assume that $\pe_{K(x)}\big((\jac H)|_{x=v}, v\big) \ge 1$ for some nonzero 
$v \in K(x)$.
Then there exists a $w \in K(x)^n$ such that $(\jac H)|_{x=v} \cdot w = v$. Since 
$\jac H \cdot y = (\jac H)|_{x=y} \cdot x$ for quadratic homogeneous $H$,
it follows that $(\jac H)|_{x=w} \cdot v = v$. So 
$$
v = (\jac H)|_{x=w}^n \cdot v = 0 \cdot v = 0
$$
Contradiction, so $\pe_{K(x)}\big((\jac H)|_{x=v}, v\big) = 0$. 
\end{proof}

\begin{theorem}
Assume $H \in K[x]^n$ such that $\jac H^n = 0$ and $\pe_{K(x)}(\jac H, x) \ge 1$. 
Then the rows of $\jac H$ are linearly independent over $K$
\end{theorem}

\begin{proof}
Since $x$ has a preimage under $y \mapsto \jac H \cdot y$, every dependence
between the rows of $\jac H$ is a dependence between the components of
$x$ as well. But the components of $x$ are linearly independent over $K$.
\end{proof}

\begin{corollary}
Assume $H \in K[x]^n$ such that $\jac H^{n-1} \cdot x = 0 = \jac H^n$ and 
$\rk \jac H = n-1$. Then the rows of $\jac H$ are linearly independent over $K$.
\end{corollary}

\begin{proof}
Since $\jac H$ is nilpotent of corank $1$, 
$\ie_{K(x)}(\jac H,x) + \pe_{K(x)}(\jac H,x) = n-1$ follows.
From $\jac H^{n-1} = 0$ we obtain $\ie_{K(x)}(\jac H,x) < n-1$, so 
$\pe_{K(x)}\allowbreak (\jac H,x) > 0$. Now apply the above theorem.
\end{proof}

\section{Dimension 4 with nilpotency}

\begin{theorem} \label{triangLK}
Let $H \in K[x]^n$, such that $\jac H$ is nilpotent. Let $L$ be an extension
field of $K$.

If $\jac H$ is similar over $L$ to a triangular
matrix, then $\jac H$ is similar over $K$ to a triangular
matrix.
\end{theorem}

\begin{proof}
Since (1) of \cite[Cor.\@ 2.2]{MR3177043} does not depend on the base field,
it follows from \cite[Cor.\@ 2.2]{MR3177043} that (2) of 
\cite[Cor.\@ 2.2]{MR3177043} does not depend on the base field either. 
By taking $r = n$, we see that $\jac H$ is similar over $K$ to a triangular
matrix, if and only if $\jac H$ is similar over $L$ to a triangular
matrix.
\end{proof}

\begin{theorem} \label{Jordan5}
Let $H \in K[x]^n$, such that $\jac H^s = 0$, but $\jac H^{s-1} \cdot x \ne 0$. 
Then there exists an extension field $L$ of $K$, such that corollary \ref{Jordancor} 
holds if we replace $K$ by $L$. 

Suppose in addition that $\rk \jac H = s-1$.
Take any $T \in \GL_n(L)$ and define $\tilde{H} := T^{-1} H(Tx)$. Let $f$
be the square-free part of any $L$-linear combination of the components of
$\jac \tilde{H}^{s-1} \cdot x$ which is nonzero. 

Then corollary \ref{Jordancor} holds over $K$ as well in the following cases:
\begin{enumerate}[\upshape (i)]

\item $\#K \ge \deg f + 1$;

\item $f$ is homogeneous and $\#K \ge \deg f$.
 
\end{enumerate}
\end{theorem}

\begin{proof}
Since the conditions of corollary \ref{Jordancor} are fulfilled up to assumptions
on the cardinality of $K$, an $L$ as given indeed exists.

From \cite[Lemma 5.1]{1310.7843}, it follows that there exists a $v \in K^n$
such that $f(T^{-1}v) \ne 0$. As $f$ is the square-free part of an $L$-linear 
combination of $T \cdot \jac \tilde{H}^{s-1} \cdot x$, $f(T^{-1}v)$ is
the square-free part of an $L$-linear combination of
\begin{align*}
T \big((\jac \tilde{H})^{s-1} x \big)\big|_{x=T^{-1}v} 
&= \big((T \cdot \jac \tilde{H} \cdot T^{-1})^{s-1} \cdot Tx\big) \big|_{x=T^{-1}v} \\
&= \big((\jac H)^{s-1} \cdot x\big)|_{x=v}
\end{align*}
Consequently, $\big((\jac H)^{s-1} \cdot x\big)\big|_{x=v} \ne 0$ along with $f(T^{-1}v)$.

Since $\big((\jac H)^{s-1} \cdot x\big)\big|_{x=v} \ne 0$,
$(\jac H)^{s-1} \cdot x \ne 0$ as well. As $\rk \jac H = s-1$,
this is only possible if the Jordan Normal 
form of $\jac H$ has one block of size $s$ and $n - s$ blocks of size $1$.
Furthermore, $M := (\jac H)|_{x=v}$ has the same Jordan Normal Form as $\jac H$,
$\ie_{K(x)}(\jac H,x) = s-1 = \ie(M,v)$ and $\pe_{K(x)}(\jac H,x) = 0 = \pe(M,v)$.
The rest of the proof is similar to the end of the proof of theorem \ref{Jordanth}.
\end{proof}

\begin{theorem} \label{qt}
Let $H \in K[x]^n$ be of degree $d$, such that $\jac H \cdot H = 0$. 
Suppose that the characteristic of $K$ is either zero or larger than $d$.

Then $H(x+tH) = H$. In particular, $\jac H$ is nilpotent. Furthermore,
$\trdeg_K K(H) \le \max\{n-2,1\}$ if $H$ is homogeneous.
\end{theorem}

\begin{proof}
The proofs of \cite[Prop.\@ 1.3]{1501.05168} and 
\cite[Lems.\@ 1.1 and 1.2]{1501.05168} for characteristic zero are still 
valid if $\deg f \le d$. Consequently, $H(x+tH) = H$ and $\jac H$ is nilpotent. 

So assume that $H$ is homogeneous. Let $f$ be a divisor of
$H_i$ for some $i$. Then $f(x+tH) \mid H_i(x+tH) = H_i$. Hence
$\deg_t f(x+tH) = 0$ and $f(x+tH) - f = 0$. Since $H$ is homogeneous,
it follows that $f$ is homogeneous as well. Consequently, we can look 
at the leading coefficient of $t$ in $f(x+tH) - f$, to deduce that
$f(H) = 0$.

Now suppose that $\trdeg_K K(H) \ge 2$. Then there are two
polynomials $f \in K[y]$ as in the above paragraph, which are 
relatively prime. Since $K[y]$ is a unique factorization domain, we
deduce that the ideal generated by these two polynomial has height
at least two, so $\trdeg_K K(H) \le n-2$. Hence
$\trdeg_K K(H) \le \max\{n-2,1\}$.
\end{proof}

\begin{theorem} \label{cubhmgdim4th}
Suppose that $\frac16 \in K$. Let $n = 4$ and $H \in K[x]^4$ be 
cubic homogeneous, such that $\jac H$ is nilpotent.

If $\jac H$ is not similar over $K$ to a 
triangular matrix, then there exists a $T \in \GL_n(K)$ such that for 
$\tilde{H} := T^{-1} H(Tx)$,
\begin{align*}
\tilde{H}_1 &- (x_1x_3x_4-x_2x_4^2) \in K[x_3,x_4]  \\
\tilde{H}_2 &- (x_1x_3^2-x_2x_3x_4) \in K[x_3,x_4]  \\
\tilde{H}_3 &\in K[x_4] \qquad \mbox{and} \qquad \tilde{H}_4 = 0
\end{align*}
Furthermore, $x + H$ is invertible and $(x + H,x_5)$ is even tame.
\end{theorem}

\begin{proof}
The case where $\rk \jac H \le 2$ follows from theorem \ref{uporkle2},
so assume that $\rk \jac H = 3$. 
Using theorems \ref{Jordanth}, \ref{triangLK} and \ref{Jordan5},
the cases $\ie_{K(x)}(\jac H, x) = 3$ and $\ie_{K(x)}(\jac H, x) = 2$ have been computed
with Maple 8, see {\tt dim4cub.pdf}.

So assume that $\ie_{K(x)}(\jac H, x) \le 1$. Then $\jac H \cdot H = 
\frac13 (\jac H)^2 \cdot x = 0$. On account of theorem \ref{qt},
$\trdeg_K K(H) \le 4-2 = 2$. This contradicts $\rk \jac H = 3$
by way of $\rk \jac H \le \trdeg_K K(H)$.

The last claim follows in a similar manner as the last claim of theorem 
\ref{uporkle2}.
\end{proof}

\begin{corollary} \label{cubhmgdim4cor}
Suppose that $\frac16 \in K$. Let $n = 3$ and $H \in K[x]^3$ be 
cubic, such that $\jac H$ is nilpotent.

If $\jac H$ is not similar over $K$ to a 
triangular matrix, then there exists a $T \in \GL_n(K)$ such that for 
$\tilde{H} = T^{-1} H(Tx)$,
\begin{align*}
\tilde{H}_1 &- (x_1x_3-x_2)\hphantom{x_3} \in K[x_3]  \\
\tilde{H}_2 &- (x_1x_3^2-x_2x_3) \in K[x_3]  \\
\tilde{H}_3 &\in K
\end{align*}
Furthermore, $x + H$ is invertible and $(x + H,x_4)$ is even tame.
\end{corollary}

\begin{proof}
To prove the first claim, we distinguish two cases:
\begin{itemize}
 
\item \emph{$H_1$, $H_2$ and $H_3$ are linearly independent over $K$.}

Define 
$$
G := \bigg(x_4^3 H\Big(\frac{x_1}{x_4},\frac{x_2}{x_4},\frac{x_3}{x_4}\Big),0\bigg)
$$
Then $G$ is cubic homogeneous and $\jac_{x,x_4} G$ is nilpotent.
From theorem \ref{cubhmgdim4th}, it follows that there exists a 
$\tilde{T} \in \GL_4(K)$, such that for 
$\tilde{G} := \tilde{T}^{-1}G\big(\tilde{T}(x,x_4)\big)$,
either 
\begin{align*}
\tilde{G}_1 &- (x_1x_3x_4-x_2x_4^2) \in K[x_3,x_4]  \\
\tilde{G}_2 &- (x_1x_3^2-x_2x_3x_4) \in K[x_3,x_4]  \\
\tilde{G}_3 &\in K[x_4] \qquad \mbox{and} \qquad \tilde{H}_4 = 0
\end{align*}
or $\jac \tilde{G}$ is an upper triangular matrix.

Take $v \in K^4$, such that $\tilde{T} v = e_4$. 
Since $H_1$, $H_2$ and $H_3$ are linearly independent over $K$, the last
row of $\tilde{T}^{-1}$ is dependent on $(0~0~0~1)$. Hence the last row 
of $\tilde{T}$ is dependent on $(0~0~0~1)$ as well. Consequently, $v_4 \ne 0$.

Using Maple or something, one can show that without affecting the formulas for 
$\tilde{G}$, we can replace $\tilde{T}$ by
$$
\tilde{T} \cdot \left( \begin{array}{cccc}
v_4 & 0 & 0 & v_1 \\ v_3 & \frac1{v_4} & 0 & v_2 \\
0 & 0 & \frac1{v_4} & v_3 \\ 0 & 0 & 0 & v_4
\end{array} \right) = \left( \begin{array}{cccc}
 & & & 0 \\ & T & & 0 \\ & & & 0 \\ 0 & 0 & 0 & 1 \\
\end{array} \right)
$$
for some $T \in \GL_3(K)$. Now one can verify that
$$
\tilde{H} = T^{-1}H(Tx) = 
\big(\tilde{G}_1(x,1),\tilde{G}_2(x,1),\tilde{G}_3(x,1)\big)
$$
and that $T$ satisfies the claims in corollary \ref{cubhmgdim4cor}.

\item \emph{$H_1$, $H_2$ and $H_3$ are linearly dependent over $K$.}

Then we may assume that $H_3 = 0$. If $1$ can be written as a 
$K$-linear combination of $H_1$, $H_2$, then we may assume that $H_2 = 1$,
which results in that $\jac H$ is an upper triangular matrix.

So assume that $1$ cannot be written as a $K$-linear combination of 
$H_1$ and $H_2$. Then we can replace $H_3$ by $1$, to obtain the
above case where $H_1$, $H_2$ and $H_3$ are linearly independent over $K$.

\end{itemize}
The last claim follows in a similar manner as the last claim of 
theorem \ref{cubhmgdim4th} and theorem \ref{uporkle2}.
\end{proof}

In 1994, Engelbert Hubbers presented a computation of all cubic homogeneous 
polynomial maps $H$ for which $\jac H$ is nilpotent, but only over fields 
of characteristic zero, see \cite{Hub94}.

\section{Quartic maps in dimension 3}

\begin{theorem} \label{qrthmgdim3th}
Let $n = 3$ and $H \in K[x]^3$ be quartic homogeneous, such that $\jac H$
is nilpotent. 

If $\frac16 \in K$, then $\jac H$ is similar over $K$ 
to a triangular matrix, and $\rk \jac H = \trdeg_K K(H)$.
\end{theorem}

\begin{proof}
Suppose that $\frac16 \in K$. We distinguish three cases.
\begin{compactitem}

\item $\ie_{K(x)}(\jac H,x) = 2$. 

Then $\rk \jac H = 2$. Using corollary \ref{Jordancor} and theorem 
\ref{triangLK}, it has been computed with Maple 8 that $\jac H$ is 
similar over $K$ to a triangular matrix, see 
{\tt dim3qrt.pdf}. Hence the components 
of $H$ are linearly dependent. So $\trdeg_K K(H) \le 2$.
From $\rk \jac H \le \trdeg_K K(H)$, we deduce that 
$\rk \jac H = 2 = \trdeg_K K(H)$.

\item $\ie_{K(x)}(\jac H,x) = 1$. 

Then $\rk \jac H \ge 1$ and
$\jac H \cdot H = \frac14 (\jac H)^2 \cdot x = 0$. 
On account of theorem \ref{qt},
$\trdeg_K K(H) \le 3-2 = 1$. From theorem \ref{trdeg1}, 
it follows that $\jac H$ is similar over $K$ to a triangular 
matrix. From $\rk \jac H \le \trdeg_K K(H)$, we deduce that
$\rk \jac H = 1 = \trdeg_K K(H)$.

\item $\ie_{K(x)}(\jac H,x) = 0$. 

Then $H = \frac14 \jac H 
\cdot x = 0$. So $\deg H < 4$. \qedhere

\end{compactitem}
\end{proof}

\begin{corollary} \label{qrthmgdim3cor}
Let $n = 3$ and $H \in K[x]^3$ be quartic, such that $\jac H$
is nilpotent. If $\frac16 \in K$, then there exists a $T \in \GL_3(K)$
such that for $\tilde{H} := T^{-1} H(Tx)$, either 
$$
\tilde{H} = \big(0, \tfrac14 x_1^4, x_1^3 x_2 + u_3 x_1^2 x_2^2 + v_3 x_1 x_2^3 + w_3 x_2^4\big)
$$
for certain $u_3, v_3, w_3 \in K$, or
$$
\tilde{H} = \big(0, \tfrac14 x_1^4 + u_2 x_1^2 x_3^2 + v_2 x_1 x_3^3 + w_2 x_3^4, 0\big)
$$
for certain $u_2, v_2, w_2 \in K$.
\end{corollary}

\begin{proof}
Suppose that $\frac16 \in K$. We distinguish two cases.
\begin{itemize}

\item $\rk \jac H = 2$.

On account of theorem \ref{triangLK}, we may assume 
without loss of generality that $\jac H$ is lower triangular.
Since $\rk \jac H = 2$, we deduce that $\parder{}{x_1} H_2$
and $\parder{}{x_2} H_3$ are both nonzero.

Since $K$ has at least $5$ elements, it follows from 
\cite[Lemma 5.1 (i)]{1310.7843} or \cite[Lemma 5.1 (ii)]{1310.7843}
that there exists a vector $\tilde{v} \in K^3$, such that 
$x_1 \parder{}{x_2} H_3$ does not vanish on $\tilde{v}$. 
As $x_1^4$ is the only term of $\parder{}{x_1} H_2$,
both $\parder{}{x_1} H_2$ and $\parder{}{x_2} H_3$ do not
vanish on $\tilde{v}$. 

Furthermore, $\tilde{v}_1 \neq 0$ and one can verify that 
$\big((\jac H)^2 \cdot x\big)\big|_{x=\tilde{v}} \neq 0$.
From the proof of theorem \ref{Jordan5}, we infer that
there exists a $T \in \GL_3(K)$
such that for $\tilde{H} := T^{-1} H(Tx)$,
$$
\big(\jac \tilde{H}\big)\big|_{x=e_1} = \left( \begin{array}{ccc}
0 & 0 & 0 \\ 1 & 0 & 0 \\ 0 & 1 & 0 \end{array} \right)
$$
Since $\jac \tilde{H}$ is similar over $K$ to a triangular matrix,
both the rows and the columns of $\jac \tilde{H}$ are dependent over $K$.
This is only possible if both the first row and the last column of 
$\jac \tilde{H}$ are zero.

As $\tr \jac \tilde{H} = 0$, the entry in the middle of $\jac \tilde{H}$
is zero as well. Consequently,
$$
\tilde{H} = \big(0, \tfrac14 x_1^4, x_1^3 x_2 + u_3 x_1^2 x_2^2 + v_3 x_1 x_2^3 + w_3 x_2^4\big)
$$
for certain $u_3,v_3,w_3 \in K$.

\item $\rk \jac H = 1$.

Since $K$ has at least $5$ elements, it follows from 
\cite[Lemma 5.1 (i)]{1310.7843} or \cite[Lemma 5.1 (ii)]{1310.7843}
that there exists a vector $\tilde{v} \in K^3$, such that 
$H$ does not vanish on $\tilde{v}$. 

Hence
$\big((\jac H) \cdot x\big)\big|_{x=\tilde{v}} = 4 H(\tilde{v}) \neq 0$.
From the proof of theorem \ref{Jordan5}, we infer that
there exists a $T \in \GL_3(K)$
such that for $\tilde{H} := T^{-1} H(Tx)$,
$$
\big(\jac \tilde{H}\big)\big|_{x=e_1} = \left( \begin{array}{ccc}
0 & 0 & 0 \\ 1 & 0 & 0 \\ 0 & 0 & 0 \end{array} \right)
$$

From theorem \ref{qrthmgdim3th}, we obtain that
$\trdeg_K K(\tilde{H}) \le 3-2 = 1$. Since $\tilde{H}$ 
is homogeneous of positive degree, we even have 
$\trdeg_K K(t\tilde{H}) = 1$. Consequently, we can deduce from 
\cite[Th.\@ 2.7]{1501.06046} or \cite[Cor.\@ 3.2]{1501.06046} that 
$\tilde{H}_1 = \tilde{H}_3 = 0$. 

As $\tr \jac \tilde{H} = 0$, the entry in the center of $\jac \tilde{H}$ 
is zero as well. Consequently,
$$
\tilde{H} = \big(0, \tfrac14 x_1^4 + u_2 x_1^2 x_3^2 + v_2 x_1 x_3^3 + w_2 x_3^4, 0\big)
$$
for certain $u_2, v_2, w_2 \in K$. \qedhere

\end{itemize}
\end{proof}

\begin{theorem} \label{qrtdim3th}
Suppose that $\frac16 \in K$.
Let $n = 3$ and $H \in K[x]^3$ be quartic, such that $\jac H$
is nilpotent. 

If $\jac H$ is not similar over $K$ to a triangular matrix, 
then there exists a $T \in \GL_3(K)$ 
such that for $\tilde{H} := T^{-1}H(Tx)$, one of the following statements holds:
\begin{enumerate}[\upshape(1)]

\item $\tilde{H}_3 \in K$ and $\deg_{x_1,x_2} \tilde{H} = 1$;

\item $\tilde{H} = \big((x_2-x_1^2),2x_1(x_2-x_1^2)+x_3,-(x_2-x_1^2)^2\big)$.

\end{enumerate}
Furthermore $(x + H)$ is invertible and $(x + H,x_4)$ is even tame.
\end{theorem}

\begin{proof}
From corollary \ref{qrthmgdim3cor}, it follows that we may assume that the quartic
part of $H$ is either
$$
\big(0, \tfrac14 x_1^4, x_1^3 x_2 + u_3 x_1^2 x_2^2 + v_3 x_1 x_2^3 + w_3 x_2^4\big)
$$
or
$$
\big(0, \tfrac14 x_1^4 + u_2 x_1^2 x_3^2 + v_2 x_1 x_3^3 + w_2 x_3^4, 0\big)
$$
Using this assumption, theorem \ref{qrtdim3th} except the last claim has been 
verified with Maple 8, see {\tt dim3upoqrt.pdf}.

To prove the last claim, suppose first that $\tilde{H}$ is as in (1).
Since $\jac \tilde{H}$ is nilpotent and $\tilde{H}_3 = 0$, we see that
$\jac_{x_1,x_2} \tilde{H}$ is nilpotent. Since $\jac_{x_1,x_2} \tilde{H} \in 
\Mat_2(K[x_3])$, it follows from lemma \ref{2x2} that there are
$a,b,c \in K[x_3]$, such that 
$$
\jac_{x_1,x_2} \tilde{H} = c \left( \begin{array}{cc} 
ab & -b^2 \\ a^2 & -ab \end{array}\right)
$$
Consequently,
\begin{align*}
\tilde{H}_1 &- cb(a x_1 - b x_2) \in K[x_3] \\
\tilde{H}_2 &- ca(a x_1 - b x_2) \in K[x_3] \\
\tilde{H}_3 &\in K
\end{align*}
So $(x+\tilde{H},x_4)$ is tame, if and only if
$$
\big(x_1 + cb(a x_1 - b x_2), x_2 + ca(a x_1 - b x_2), x_3, x_4\big)
$$
is tame, which follows from lemma \ref{tamelm} by way of extension of scalars.

Suppose next that $\tilde{H}$ is as in (2). From lemma \ref{tamelm2} at the end of this
section, with $c = 1$, it follows that $x + \tilde{H}$ is tame.
\end{proof}

\begin{corollary} \label{qrtdim3cor}
Suppose that $\frac16 \in K$.
Let $n = 4$ and $H \in K[x]^3$ be quartic homogeneous, such that $\jac H$ is
nilpotent. 

Suppose that $H_1, H_2, H_3, H_4$ are linearly dependent over $K$. 

If $\jac H$ is not similar over $K$ to a triangular matrix, 
then there exists a $T \in \GL_4(K)$ 
such that for $\tilde{H} := T^{-1}H(Tx)$, one of the following statements holds:
\begin{enumerate}[\upshape(1)]

\item $\tilde{H}_3 \in K$, $\tilde{H}_4 = 0$ and $\deg_{x_1,x_2} \tilde{H} = 1$;

\item $\tilde{H} = \big(x_4^2(x_2x_4-x_1^2),2x_1x_4(x_2x_4-x_1^2)+x_3x_4^3,-(x_2x_4-x_1^2)^2,0\big)$.

\end{enumerate}
Furthermore $x + H$ is invertible and $(x + H,x_5)$ is even tame.
\end{corollary}

\begin{proof}
Since $H_1, H_2, H_3, H_4$ are linearly dependent over $K$, we may assume without
loss of generality that $H_4 = 0$. Define $G$ by 
$$
G := \big(H_1(x_1,x_2,x_3,1),H_2(x_1,x_2,x_3,1),H_3(x_1,x_2,x_3,1)\big)
$$
Then $\jac_{x_1,x_2,x_3} G$ is nilpotent, so $G$ is as $H$ in theorem \ref{qrtdim3th}.
Consequently, there exists a $\tilde{T} \in \GL_3(K)$, such that for 
$\tilde{G} = \tilde{T}^{-1}G(\tilde{T}x)$, one of the following statements holds:
\begin{enumerate}[\upshape(1)]

\item $\tilde{G}_3 \in K$ and $\deg_{x_1,x_2} \tilde{G} = 1$;

\item $\tilde{G} = \big((x_2-x_1^2),2x_1(x_2-x_1^2)+x_3,-(x_2-x_1^2)^2\big)$;

\item $\jac_{x_1,x_2,x_3} \tilde{G}$ is an upper triangular matrix.

\end{enumerate}
Let 
$$
T := \left(\begin{array}{cccc}
& & & 0 \\
& \tilde{T} & & 0 \\
& & & 0 \\
0 & 0 & 0 & 1
\end{array} \right)
$$
Then one can verify that
$$
\tilde{H} = T^{-1}H(Tx) = x_4^4 
\bigg(\tilde{G}\Big(\frac{x_1}{x_4},\frac{x_2}{x_4},\frac{x_3}{x_4}\Big),0\bigg)
$$
and that $T$ satisfies the claims in corollary \ref{qrtdim3cor}.

The proof of the last claim is similar to that of the last claim of theorem \ref{qrtdim3th}.
\end{proof}

\begin{lemma}\label{tamelm2}
\begin{align*}
\lefteqn{\big(x_1+c^2(cx_2-x_1^2), x_2+2cx_1(cx_2-x_1^2)+x_3c^3, x_3-(cx_2-x_1^2)^2\big)} \\
&= \big(x_1, x_2+c^3x_3, x_3\big) \circ {} \\
&\quad~ \big(x_1+c^2(cx_2-x_1^2),  x_2+ 2cx_1(cx_2-x_1^2) + c^3(cx_2-x_1^2)^2, x_3\big) \circ {} \\
&\quad~ \big(x_1,x_2,x_3-(cx_2-x_1^2)^2\big)
\end{align*}
and
\begin{align*}
\lefteqn{\big(x_1+c^2(cx_2-x_1^2), x_2+2cx_1(cx_2-x_1^2)+c^3(cx_2-x_1^2)^2, x_3\big)} \\
&= \big(c^2x_3+x_1, c^3x_3^2+2cx_1x_3+x_2, x_3\big) \circ 
   \big(x_1, x_2, cx_2-x_1^2+x_3\big) \circ {} \\
&\quad~ \big(-c^2x_3+x_1, c^3x_3^2-2cx_1x_3+x_2, x_3\big) \circ 
   \big(x_1, x_2, x_3-cx_2+x_1^2\big)
\end{align*}
\end{lemma}

One can verify lemma \ref{tamelm2} with Maple or something, or do
the first equality by hand, and use
the proposition in \cite{MR1001475}, with 
$$
D = c^2\parder{}{x_1} + 2cx_1\parder{}{x_2} \qquad \mbox{and} \qquad  
c x_2 - x_1^2 \in \ker D
$$
to get a proof of the second equality.


\bibliographystyle{cubhmgrk2j}
\bibliography{cubhmgrk2j}

\end{document}